\documentclass[12pt]{article}
\usepackage{amssymb}
\usepackage{verbatim}
\usepackage{array}
\usepackage{latexsym}
\usepackage{enumerate}
\usepackage{amsmath}
\usepackage{amsfonts}
\usepackage{amsthm}
\usepackage{color}
\usepackage{hyperref}
\usepackage[english]{babel}
\usepackage{cancel}

\newtheorem{theorem}{Theorem}[section]
\newtheorem{proposition}[theorem]{Proposition}
\newtheorem{lemma}[theorem]{Lemma}
\newtheorem{definition}[theorem]{Definition}
\newtheorem{corollary}[theorem]{Corollary}

\newtheorem{remark}[theorem]{Remark}

{\theoremstyle{definition}
\newtheorem*{definition*}{Definition}

\newtheorem*{proposition*}{Proposition}
\newtheorem*{corollary*}{Corollary}
\newtheorem*{lemma*}{Lemma}

\def\C{\mathcal C}

\def\cC{\mathcal C}

\def\cG{\mathcal G}

\def\cL{\mathcal L}

\def\PG{{\rm{PG}}}

\def\deg{\mbox{\rm deg}}
\def\det{\mbox{\rm det}}

\def\min{{\rm min}}

\def\dim{\mbox{\rm dim}}

\def\fq{{\mathbb F}_q}
\def\fqn{{\mathbb F}_{q^n}}
\def\fqnt{{\mathbb F}_{q^{nt}}}

\def\ff{{\underline{f}}}

\sloppy

\title{Linear sets from projection of Desarguesian spreads}

\author{Vito Napolitano, Olga Polverino, Giovanni Zini and Ferdinando Zullo}

\date{}

\begin{document}
\maketitle


\begin{abstract}
Every linear set in a Galois space is the projection of a subgeometry, and most known characterizations of linear sets are given under this point of view.
For instance, scattered linear sets of pseudoregulus type are obtained by considering a Desarguesian spread of a subgeometry and projecting from a vertex which is spanned by all but two director spaces.
In this paper we introduce the concept of linear sets of $h$-pseudoregulus type, which turns out to be projected from the span of an arbitrary number of director spaces of a Desarguesian spread of a subgeometry.
Among these linear sets, we characterize those which are $h$-scattered and solve the equivalence problem between them; a key role is played by an algebraic tool recently introduced in the literature and known as Moore exponent set.
As a byproduct, we classify asymptotically $h$-scattered linear sets of $h$-pseudoregulus type.
\end{abstract}



\thanks{This research was partially supported by the Italian National Group for Algebraic and Geometric Structures and their Applications (GNSAGA - INdAM). The first and the last author were also supported by the project ''VALERE: Vanvitelli pEr la RicErca" of the University of Campania ''Luigi Vanvitelli''.}

\section{Introduction}

Let $q$ be a prime power, $\fq$ be the finite field of order $q$, and $\fqn$ be the degree $n$ extension of $\fq$.
$\fq$-linear sets in a Galois space $\Lambda=\PG(r-1,q^n)$ are a natural generalization of subgeometries $\Sigma=\PG(s,q)$ of $\Lambda$, and have been intensively investigated in the last decades.
The intrinsic theoretical interest towards linear sets has been boosted in the last years by their applications in a number of areas, where they are used to construct or characterize a wide variety of geometrical and combinatorial objects.
These areas include, but are not limited to, linear codes, semifields, blocking sets, strongly regular graphs; see \cite{Lavrauw,LVdV2015,Polverino} and the references therein.

The aim of this paper is to investigate a family of linear sets of $\Lambda=\PG((h+1)t-1,q^n)$, with $h,t \geq 1$, obtained by considering a Desarguesian $(n-1)$-spread $\mathcal{D}$ of $\Sigma=\PG(nt-1,q)$ and projecting from the span of $n-h-1$ director spaces of $\mathcal{D}$.
When $h=1$, maximum scattered $\fq$-linear sets of pseudoregulus type have been characterized in these terms in \cite{LMPT:14}, generalizing results of \cite{LVdV2013,MPT2007}. In particular, those linear sets are associated with exactly $\frac{q^{nt}-1}{q^n-1}$ pairwise disjoint lines (forming the pseudoregulus) and admit exactly two $(t-1)$-dimensional subspaces intersecting each of these lines in one point. When $h\geq 2$, linear sets of this type still have a similar structure, once we replace lines by $h$-dimensional subspaces.
These subspaces form what we call a $h$-\emph{pseudoregulus}, cf. Definition \ref{def:hpseudo}; they admit exactly $h+1$ distinct $(t-1)$-subspaces intersecting each element of the $h$-pseudoregulus in one point.
Recently, in \cite{CMPZu} a family of scattered linear sets has been detected by imposing further requests on their intersections with $h$-subspaces of $\Lambda$; they have been named \emph{$h$-scattered} linear sets.
Our main results are the following ones:
\begin{itemize}
    \item we show examples of linear sets of $h$-pseudoregulus type, for any admissible value of $h$, $t$, $q$ and $n$ (Theorem \ref{th:fromalgebratogeometry});
    \item we provide an analytic description of all linear sets of this type (Theorem  \ref{thm:maintheoremsec3});
    \item we establish a connection between $h$-scattered linear sets of $h$-pseudoregulus type and a recently introduced algebraic tool called Moore exponent set (Theorem \ref{th:mainsec4});
    \item we classify asymptotically maximum $h$-scattered linear sets of $h$-pseudoregulus type (Theorem \ref{th:asymp}).
\end{itemize}
The paper is organized as follows.
Section \ref{sec:preliminaries} contains preliminary results and tools about Desarguesian spreads, linear sets, rank metric codes and Moore exponent sets.

Section \ref{sec:hpseudo} is devoted to the description of linear sets of $h$-pseudoregulus type, as defined in Definition \ref{def:hpseudo}. In particular, we introduce in Theorem \ref{th:fromalgebratogeometry} a class of linear sets of $h$-pseudoregulus type for every $h\geq1$, $t,n\geq2$, and prime power $q$; we prove in Theorem \ref{thm:maintheoremsec3} that, up to equivalence, all linear sets of $h$-pseudoregulus type are of this form.

In Section \ref{sec:maxhscatt} we deal with those linear sets of $h$-pseudoregulus type which are maximum $h$-scattered.
More precisely, in Theorem \ref{th:mainsec4} we associate with each of them a Moore exponent set for $q$ and $n$; among the linear sets of $h$-pseudoregulus type, the maximum $h$-scattered ones are actually characterized by means of Moore exponent sets.
Moreover, in Theorem \ref{th:uniqPseudo} we show the uniqueness of the $h$-pseudoregulus for maximum $h$-scattered linear sets which are not subgeometries.
As a consequence, we completely solve in Corollary \ref{cor:equivissue} the equivalence issue for those linear sets.
Finally, relying on the asymptotic results in \cite{DANYU} for Moore exponent sets, we give in Theorem \ref{th:asymp} an asymptotic classification of maximum $h$-scattered linear sets of $h$-pseudoregulus type.

\section{Preliminary notions and results}\label{sec:preliminaries}

\subsection{Desarguesian spreads}\label{sec:desspreads}

An $(n-1)$-spread of $\mathrm{PG}(n t -1, q)$  is a family $\mathcal{S}$ of mutually disjoint $(n-1)$--dimensional subspaces such that each point of   $\mathrm{PG}(n t -1, q)$ belongs to an element of $\mathcal{S}$. Examples of spreads are the Desarguesian ones, which can be constructed as follows.   
Every point $P$ of $\mathrm{PG}(n-1, q^n)= \mathrm{PG}(V, \mathbb{F}_{q^n})$ defines an $(n-1)$--dimensional subspace $X(P)$ of $\mathrm{PG}(n t -1, q) = \mathrm{PG}(V, \mathbb{F}_{q})$ and ${\mathcal D} = \{X(P)\, : \, P\in \mathrm{PG}(t  -1, q^n)\}$ is a spread of $\mathrm{PG}(n t -1, q)$, called a {\em Desarguesian spread} (see \cite[Section 25]{Segre}). If $t > 2$, the incidence structure $\Pi_{t-1}({\mathcal D})$, whose points are the elements of $\mathcal{D}$  and whose lines are the $(2n-1)$--dimensional subspaces of $\mathrm{PG}(n t -1, q)$ joining two distinct elements  of $\mathcal{D}$, is isomorphic to $\mathrm{PG}(t  -1, q^n)$. The structure $\Pi_{t-1}({\mathcal D})$  is called the $\mathbb{F}_q$--{\em linear representation of} $\mathrm{PG}(t-1, q^n)$.

A way to obtain a Desarguesian $(n-1)$--spread of $\mathrm{PG}(n t - 1, q)$ is the following. Let $\Sigma = \mathrm{PG}(nt -1, q)$ and $\Sigma^* = \mathrm{PG}(nt -1, q^n)$. Embed $\Sigma$ in $\Sigma^*$ in such a way that $\Sigma=\mathrm{Fix}(\Psi)$, where $\Psi\in{\rm P\Gamma L}(\Sigma^*)$ is a semilinear collineation of order $n$ of $\Sigma^*$ whose fixed points are the points of $\Sigma$.
\begin{lemma}{\rm \cite[Lemma 1]{Lun99}}\label{lemma:lunardon}
Let $S$ be a subspace of $\Sigma^*$.
Then $\dim\, S=\dim(S\cap\Sigma)$ if and only if $S=S^{\Psi}$.
\end{lemma}
Let $\Theta =\mathrm{PG}(t-1, q^n)$ be a subspace of $\Sigma^*$ such that $\Theta, \Theta^\Psi, \ldots, \Theta^{\Psi^{n-1}}$ span the whole space $\Sigma^*$. If $P$ is a point  of $\Theta$, then, by Lemma \ref{lemma:lunardon}, $X^*(P)= \langle P, P^\Psi, \ldots, P^{\Psi^{n-1}}\rangle$ is a $(n-1)$-dimensional subspace of $\Sigma^*$ defining a $(n-1)$-subspace $X(P) = X^*(P) \cap \Sigma$ of $\Sigma$. As $P$ runs over the subspace $\Theta$ we get a set of $q^{n(t-1)} + q^{n(t-2)}+ \cdots + q^n + 1$ mutually disjoint $(n-1)$-dimensional subspaces of $\Sigma$. Such a set is denoted by ${\mathcal D } = {\mathcal D}(\Theta)$ and is a Desarguesian $(n-1)$-spread of $\Sigma$; see \cite{Segre}.  The $(t-1)$-dimensional spaces $\Theta, \Theta^\Psi, \ldots, \Theta^{\Psi^{n-1}}$ are uniquely defined by the Desarguesian spread $\mathcal D$, i.e. ${\mathcal D}(\Theta) = {\mathcal D}(X)$ if and only if $X= \Theta^{\Psi^i}$ for some $i\in\{0, 1,\ldots, n-1\}$, and are called {\em director spaces} of $\mathcal D$ (cf. \cite{LMPT:14}).
Proposition \ref{prop:desspreads} will play a special role in what follows.

\begin{proposition}\label{prop:desspreads}{\rm \cite[Remark 2.1]{LMPT:14}}
Let $\mathcal S$ be an $(n-1)$--spread of $\Sigma = \mathrm{PG}(nt-1, q)$ embedded in $\Sigma^\ast = \mathrm{PG}(nt-1, q^n)$ in such a  way that $\Sigma = Fix(\Psi)$, where $\Psi$ is a semilinear collineation of $\Sigma^\ast$ of order $n$. If $H$ is a $(t-1)$--dimensional subspace of $\Sigma^\ast$ such that 
\begin{itemize}
    \item $\Sigma^\ast = \langle H, H^\Psi, \ldots, H^{\Psi^{n-1}}\rangle$,
    \item $X^{\ast}\cap H \neq \emptyset$ for each $(n-1)$--dimensional subspace $X^\ast$ of $\Sigma^\ast$
such that $X^\ast \cap \Sigma \in \mathcal S$,
\end{itemize}
then $\mathcal{D}(H)= \mathcal S$, i.e. $\mathcal S$ is a Desarguesian spread and $H$ is one of its director spaces. 
\end{proposition}

\subsection{Linear sets}
Let $\Lambda = \mathrm{PG}(r-1, q^n)= \mathrm{PG}(V, \mathbb{F}_{q^n})$, $q=p^h$, $p$ prime. A set $L$ of points of $\Lambda$ is an {\em
$\mathbb{F}_q$-linear set} of $\Lambda$ if it is defined by the non-zero vectors of an $\mathbb{F}_q$-subspace $U$ of $V$,  i.e. $L= L_U=\{\langle u\rangle_{\fqn}\, : \, u\in U^*\}$. The linear set $L_U$ has  {\em rank} $k$ if $\dim_{\fq} U = k$. If $\Omega= \mathrm{PG}(W, \mathbb{F}_{q^n})$ is a subspace of $\Lambda$ and $L_U$ is an $\mathbb{F}_q$-linear set of $\Lambda$, then $\Omega\cap L_U$ is an $\mathbb{F}_q$--linear set of $\Omega$. The subspace $\Omega$ of $\Lambda$ has {\em weight} $i$ in $L_U$ if $\dim_{\fq}(U\cap W) = i$, and we write $\omega_{L_U}(\Omega)= i$.  If $L_U\neq \emptyset$, we have $\vert L_U\vert \equiv 1\pmod q$ and
\[
\vert L_U\vert \le q^{k-1}+q^{k-2}+ \cdots +q+1.\]
An $\mathbb{F}_q$-linear set $L_U$ of $\Lambda$ of rank $k$ is {\em scattered} if all points of $L_U$ have weight one, or equivalently, if $L_U$ has maximum size $ q^{k-1}+q^{k-2}+ \cdots +q+1$.

\begin{theorem} \label{blok}{\rm \cite[Theorem 4.2]{BLAV:00}} A scattered $\mathbb{F}_q$-linear set of $\mathrm{PG}(r-1, q^n)$ has rank at most $rn/2$.
\end{theorem}

A scattered $\mathbb{F}_q$-linear set $L$ of $\mathrm{PG}(r-1, q^n)$ of maximum rank $rn/2$ is called a {\em maximum scattered} linear set.
As a consequence of  Theorem \ref{blok} one gets that  maximum scattered $\mathbb{F}_q$-linear sets of $\PG(r-1,q^n)$ span the whole space $\PG(r-1,q^n)$.

If $\dim_{\fq} U = \dim_{\fqn} V = r$ and $\langle U\rangle_{\fqn}= V$, then the linear set  $L_U$ is a subgeometry of $\mathrm{PG}(V, \mathbb{F}_{q^n})=\mathrm{PG}(r-1, q^n)$ isomorphic to $\mathrm{PG}(r-1, q)$. If $n=2$, then $L_U$ is a Baer subgeometry of $\mathrm{PG}(r-1, q^2)$.

In \cite{LP:04} the following characterization of $\mathbb{F}_q$-linear sets is given. Let  $\Sigma=\PG(S,\fq)=\mathrm{PG}(k-1, q)$ be a subgeometry of
$\Sigma^*=\PG(S^*,\fqn)=\mathrm{PG}(k-1, q^n)$, $\Gamma=\PG(H,\fqn)$ be a $(k-r-1)$-dimensional subspace of $\Sigma^*$ disjoint from $\Sigma$, and $\Lambda=\PG(V,\fqn)=\mathrm{PG}(r-1, q^n)$ be an $(r-1)$-dimensional subspace of $\Sigma^*$ disjoint from $\Gamma$. Let $L=\{p_{\Gamma,\Lambda}(P)=\langle \Gamma, P\rangle \cap \Lambda\, : \, P\in \Sigma\}$ denote the projection of $\Sigma$ from $\Gamma$ to $\Lambda$.
The subspaces $\Gamma$ and $\Lambda$ are  the {\em center} and {\em axis} of the projection $p_{\Gamma,\Lambda}$, respectively.

\begin{theorem}\label{th:projection} {\rm \cite[Theorems 1 and 2]{LP:04}}  If $L$ is a projection of $\Sigma=\mathrm{PG}(k-1, q)$ from $\Gamma=\PG(k-r-1,q^n)$ to $\Lambda=\mathrm{PG}(r-1, q^n)$, then $L$ is an $\mathbb{F}_q$-linear set of $\Lambda$ of rank $k$ and $\langle L\rangle= \Lambda$.
Furthermore, with the above notation, $L=L_U$, where 
\begin{equation}\label{eq:rappr}
U=(S+H)\cap V.
\end{equation}
Conversely, if $L$ is an $\mathbb{F}_q$-linear set of rank $k$ and $\langle L\rangle= \Lambda$, then either $L$ is a subgeometry of $\Lambda$ or, for each $(k-r-1)$-dimensional subspace $\Gamma$ of $\Sigma^*=\mathrm{PG}(k-1, q^n)$ disjoint from $\Lambda$, there exists a subgeometry $\Sigma$ of $\Sigma^*$ disjoint from $\Gamma$ such that $L=p_{\Gamma, \Lambda}(\Sigma)$.
\end{theorem}

Proposition \ref{prop:axisimmaterial} shows that the linear set $L=p_{\Gamma, \Lambda}(\Sigma)$ does not depend on the axis $\Lambda$ of the projection, i.e. if $\overline{\Lambda}$ is an $(r-1)$-dimensional subspace of $\Sigma^*$ disjoint from $\Gamma$, then $\overline{L}=p_{\Gamma,\overline{\Lambda}}(\Sigma)$ is projectively equivalent to $L=p_{\Gamma,\Lambda}(\Sigma)$.

\begin{proposition}\label{prop:axisimmaterial}
With the same notation as in Theorem {\rm \ref{th:projection}}, let $\overline{L}=L_{\overline{U}}$ be the $\fq$-linear set obtained by projecting $\Sigma$ from $\Gamma$ to another subspace $\overline{\Lambda}=\PG(\overline{V},\fqn)=\PG(r-1,q^n)$ of $\Sigma^*$.
Then there exists an $\fqn$-linear isomorphism $\omega:V\to\overline{V}$ such that $\omega(U)=\overline{U}$ and hence $\varphi_{\omega}(L)=\overline{L}$, where $\varphi_{\omega}:\Lambda\to\overline{\Lambda}$ is the projectivity induced by $\omega$.
\end{proposition}

\begin{proof}
Since $\Lambda\cap\Gamma=\overline{\Lambda}\cap\Gamma=\emptyset$, we have $S^*=V\oplus H=\overline{V}\oplus H$ and hence the maps $\phi_V:v\in V\mapsto v+H\in S^*/H$, $\phi_{\overline{V}}:v\in\overline{V}\mapsto v+H\in S^*/H$ are $\fqn$-linear isomorphisms.
By \eqref{eq:rappr}, both the images $\phi_V(U)=U+H$ and $\phi_{\overline{V}}(\overline{U})=\overline{U}+H$ are equal to $S+H$.
Therefore, $\omega=\phi_{\overline{V}}^{-1}\circ\phi_{V}:V\to\overline{V}$ is an $\fqn$-linear isomorphism with $\omega(U)=\overline{U}$ and hence $\phi_\omega(L)=L_{\omega(U)}=\overline{L}$, i.e. $L$ and $\overline{L}$ are projectively equivalent.
\end{proof}

The concept of scattered linear sets was first generalized in \cite{Lunardon2017,ShVdV} to the concept of scattered linear sets with respect to the hyperplanes, and later in \cite{CMPZu} as in Definition \ref{def:hscattered}.

\begin{definition}\label{def:hscattered}
Let $L=L_U$ be an $\fq$-linear set in $\Lambda=PG(r-1,q^n)$, and $h\leq r$ be a positive integer.
The linear set $L$ is $h$-scattered if $\langle L\rangle=\Lambda$ and, for every $(h-1)$-dimensional subspace $\Omega$ of $\Lambda$, the weight $w_{L}(\Omega)$ is at most $h$.
If $L$ is an $(r-1)$-scattered $\fq$-linear set of $\Lambda$, then $L$ is also said to be scattered with respect to the hyperplanes.
\end{definition}

Recall from \cite[Proposition 2.1]{CMPZu} that, if $L$ is a $h$-scattered linear set, then $L$ is also an $m$-scattered linear set whenever $1\leq m\leq h$.
Note also that $1$-scattered linear sets of $\Lambda$ are exactly the scattered linear sets of $\Lambda$ which span $\Lambda$ over $\fqn$.

The following upper bound on the rank of a $h$-scattered linear set holds.

\begin{theorem}\label{th:bound}{\rm \cite[Theorem 2.3]{CMPZu}}
If $L_U$ is a $h$-scattered linear set of rank $k$ in $\Lambda=\PG(r-1,q^n)$, then one of the following holds:
\begin{itemize}
\item $k=r$ and $L_U$ is a subgeometry $\PG(r-1,q)$ of $\Lambda$;
\item $k\leq\frac{rn}{h+1}$.
\end{itemize}
\end{theorem}
A $h$-scattered linear set of maximum rank is said to be a {\em maximum $h$-scattered} linear set.
Maximum $h$-scattered linear sets can be constructed by using the following result.

\begin{theorem}\label{th:directsum}{\rm \cite[Theorem 2.5]{CMPZu}}
Let $V=V_1\oplus \ldots \oplus V_t$, let
$L_{U_i}$ be a $h_i$-scattered $\fq$-linear set in $\PG(V_i,\mathbb{F}_{q^n})$, with $i\in \{1,\ldots,t\}$ and let
\[U=U_1\oplus \ldots \oplus U_t.\] 
The $\fq$-linear set $L_U$ is $h$-scattered in $\Lambda=\PG(V,\mathbb{F}_{q^n})$, with $h=\min\{h_1,\ldots,h_t\}$.
\end{theorem}

Furthermore, the projective equivalence for $h$-scattered linear sets coincides with the equivalence of the corresponding subspaces when $h\geq 2$.

\begin{theorem}\label{th:simple}{\rm \cite[Theorem 4.5]{CMPZu}}
Let $h\geq2$. Two $h$-scattered $\fq$-linear sets $L_U$ and $L_W$ of $\mathrm{PG}(r-1,q^n)$ are
	$\mathrm{P}\Gamma\mathrm{L}(r,q^n)$-equivalent if and only if $U$ and $W$ are $\Gamma\mathrm{L}(r,q^n)$-equivalent.
\end{theorem}

\subsection{Rank metric codes}

Rank metric codes were introduced by Delsarte \cite{Delsarte} in 1978 and they have been intensively investigated in recent years for their applications; we refer to \cite{sheekey_newest_preprint} for a recent survey on this topic.
The set of $m \times n$ matrices $\fq^{m\times n}$ over $\fq$ may be endowed with a metric, called \emph{rank metric}, defined by
\[d(A,B) = \mathrm{rk}\,(A-B)\]
for any $A,B \in \fq^{m\times n}$.
A subset $\C \subseteq \fq^{m\times n}$ equipped with the rank metric is called a \emph{rank metric code} (or \emph{RM}-code for short).
The minimum distance of $\C$ is
\[d = \min\{ d(A,B) \colon A,B \in \C,\,\, A\neq B \}.\]
We are interested in RM-codes which are $\fq$-\emph{linear}, i.e. $\fq$-subspaces of $\fq^{m\times n}$.
In \cite{Delsarte}, Delsarte showed that the parameters of these codes must obey a Singleton-like bound, i.e.
\[ |\C| \leq q^{\max\{m,n\}(\min\{m,n\}-d+1)}. \]
When the equality holds, we call $\C$ a \emph{maximum rank distance} (\emph{MRD} for short) code.
Examples of $\fq$-linear MRD-codes were first found by Delsarte in \cite{Delsarte} and rediscovered by Gabidulin in \cite{Gabidulin}; although these codes have been originally found out by Delsarte, they are called \emph{Gabidulin codes} since Gabidulin's publication contributed significantly to the development of rank metric codes.
Two $\fq$-linear RM-codes $\C$ and $\C'$ are equivalent if and only if there exist $X \in \mathrm{GL}(m,q)$, $Y \in \mathrm{GL}(n,q)$ and a field automorphism $\sigma$ of $\fq$ such that
\[\C'=\{XC^\sigma Y \colon C \in \C\}.\]

Let $\cC\subseteq \fq^{m \times n}$ be an RM-code, the \emph{adjoint code} of $\C$ is
\[ \C^\top =\{C^t \colon C \in \C\}. \]
Clearly, if $\C$ is an MRD-code then $\C^\top$ is MRD.

Also, the \emph{left} and \emph{right} idealisers of $\C$ are defined as $L(\C)=\{A \in \mathrm{GL}(m,q) \colon A \C\subseteq \C\}$ and $R(\C)=\{B \in \mathrm{GL}(n,q) \colon \C B \subseteq \C\}$. They are invariant under the equivalence of rank metric codes.
Further invariants have been introduced in \cite{GiuZ,NPH2}.
In particular in \cite[Section 4]{GiuZ}, the authors introduced the \emph{Gabidulin index} of a rank metric code $\C$ as the maximum dimension of a subcode $\cG \subseteq \cC$ equivalent to a generalized Gabidulin code. The Gabidulin index has been calculated for the known families of MRD-codes with maximum left idealiser, see \cite[Theorem 4.2]{GiuZ}.

Much of the focus on MRD-codes of $\fq^{n\times n}$ to date has been on codes which are $\fqn$-\emph{linear}, i.e. codes in which the
left (or right) idealiser contains a field isomorphic to $\fqn$, since for these codes a fast decoding algorithm has been developed in \cite{Gabidulin}. Therefore, from a cryptographic point of view, it is very important to have many different examples of $\fqn$-linear MRD-codes.
Unfortunately, very few examples are known, see \cite{BZZ,CMPZa,CMPZ,CsMZ2018,Delsarte,Gabidulin,LP2001,MMZ,Sheekey2016,ZZ}.

From now on, we identify $\fq^{n\times n}$ with the algebra $\mathrm{End}_{\fq}(\fqn)$.
Since $\mathrm{End}_{\fq}(\fqn)$ is isomorphic to the ring of $q$-polynomials over $\fqn$ 
\[\cL_{n,q}=\left\{\sum_{i=0}^{n-1}a_i x^{q^i}\colon a_i \in \fqn\right\},\]
with addition and composition modulo $x^{q^n}-x$ as operations, we will consider an $\fq$-linear RM-code $\cC$ as an $\fq$-subspace of $\cL_{n,q}$.
In this setting, two $\fq$-linear MRD-codes $\cC_1$ and $\cC_2$ are  equivalent if and only if there exist two invertible $q$-polynomials  $\varphi_1$, $\varphi_2\in \cL_{n,q}$ and $\rho\in \mathrm{Aut}(\fq)$ such that
\[ \varphi_1\circ f^\rho \circ \varphi_2 \in \cC_2 \text{ for all }f\in \cC_1,\]
where $\circ$ stands for the composition of maps and $f^\rho(x)= \sum a_i^\rho x^{q^i}$ for $f(x)=\sum a_i x^{q^i}$.

The family of MRD-codes first found by Delsarte in \cite{Delsarte}, then by Gabidulin in \cite{Gabidulin} and generalized by Kshevetskiy and Gabidulin in \cite{kshevetskiy_new_2005}, can be described as follows:
\[ \mathcal{G}_{k,s}=\langle x,x^{q^s},\ldots,x^{q^{s(k-1)}} \rangle_{\fqn}\subseteq \mathcal{L}_{n,q} \]
with $k\leq n-1$ and $\gcd(s,n)=1$.

\begin{remark}
Let $\mathrm{Tr}_{q^n/q}:\fqn\to\fq$ be the trace map $x\mapsto x+x^q+\ldots+x^{q^{n-1}}$.
The \emph{adjoint} of a $q$-polynomial $f(x)=\sum_{i=0}^{n-1}a_i x^{q^i}\in\cL_{n,q}$ with respect to the bilinear form $\langle x,y\rangle=\mathrm{Tr}_{q^n/q}(xy)$
is given by
\[\hat{f}(x)=\sum_{i=0}^{n-1}a_{i}^{q^{n-i}} x^{q^{n-i}}.\]
\end{remark}

For a rank metric code $\cC$ given by a set of linearized polynomials, its adjoint code may be seen as
\[\cC^\top= \{\hat{f}\colon f\in\cC\},\] 
and the left and right idealisers of $\cC$ can be written as:
\[L(\cC)= \{ \varphi \in \cL_{n,q}\colon \varphi \circ f \in \cC \text{ for all }f\in \cC \},\]
\[R(\cC)= \{ \varphi \in \cL_{n,q}\colon f \circ \varphi \in \cC \text{ for all }f\in \cC \}.\]

In the case of RM-codes generated by monomials, the equivalence problem is completely solved as follows.
\begin{theorem}{\rm \cite[Theorem 2.3]{CMPZ}}\label{th:equivmon}
For $j=1,2$, let $I_j$ be a $k$-subset of $\{0,\ldots,n-1\}$ and
\[\mathcal{C}_j=\langle x^{q^i}\colon i\in I_j\rangle_{\fqn}.\]
Then $\mathcal{C}_1$ and $\mathcal{C}_2$ are equivalent if and only if
\[I_1=I_2+s=\{i+s\pmod n\,\colon\, i\in I_2\}\]
for some $s\in\{0,\ldots,n-1\}$.
\end{theorem}

Sheekey in \cite{Sheekey2016} proved that some of these codes are connected to \emph{maximum scattered linear sets} of the projective line; for further connections with linear sets see also \cite{CMPZu,CSMPZ2016,Lunardon2017,ShVdV}.
In particular, in \cite[Section 2.7]{Lunardon2017} and in \cite{ShVdV} the connection between $\fqn$-linear MRD-codes and maximum $(r-1)$-scattered $\fq$-linear sets in $\PG(r-1,q^n)$ has been pointed out, see also \cite[Section 4.1]{CMPZu}.

\begin{theorem} {\rm \cite[Proposition 3.5]{ShVdV}}
	\label{cod}
	$\cC$ is an $\fq$-linear MRD-code of $\cL_{n,q}$ with minimum distance $n-r+1$ and with left idealiser isomorphic to $\fqn$ if and only if, up to equivalence,
	\begin{equation}\label{eq:C}
	\cC=\langle f_1(x),\ldots,f_r(x)\rangle_{\fqn}
	\end{equation}
	for some $f_1,f_2,\ldots,f_r \in \cL_{n,q}$	and the $\fq$-subspace
	\[U_{\cC}=\{(f_1(x),\ldots,f_r(x)) \colon x\in \fqn\}\]
	defines a maximum $(r-1)$-scattered $\fq$-linear set of $\PG(\fqn^r,\fqn)$.
Furthermore, two $\fq$-linear MRD-codes $\cC$ and $\cC'$ of $\cL_{n,q}$ with minimum distance $n-r+1$ of type \eqref{eq:C} are equivalent if and only if $U_{\cC}$ and $U_{\cC'}$ are $\Gamma\mathrm{L}(r,q^n)$-equivalent.
\end{theorem}


\begin{remark}\label{rem:MRDpol}
If
\[\cC=\langle f_1(x),\ldots,f_r(x)\rangle_{\fqn} \subseteq \cL_{n,q}\]
is an MRD-code, then 
\begin{itemize}
    \item $\dim_{\fq} \ker f\leq r-1$ for each $f\in \cC\setminus\{0\}$;
    \item for each $i \in \{0,\ldots,r-1\}$, there exists $g \in \cC$ such that $\dim_{\fq} \ker g=i$ (see e.g. {\rm \cite[Lemma 2.1]{LTZ2}}).
\end{itemize}
\end{remark}

\subsection{Moore exponent sets}

Let $I =\{i_0, i_1, \ldots i_{k-1}\}\subseteq \mathbb{Z}_{\ge 0}$. For every $A = (\alpha_0, \alpha_1 \ldots, \alpha_{k-1})\in \mathbb{F}_{q^n}^k$, denote
 \[M_{A, I} = \left( \begin{array}{crcrcrcrccrcrcrcrcrcrcrcrcr}
 \alpha_0^{q^{i_0}}  &   \alpha_0^{q^{i_1}}   & \cdots \cdots  &  \alpha_0^{q^{i_{k-1}}}  \\
 \; \\
 \alpha_1^{q^{i_0}}  &   \alpha_1^{q^{i_1}}   & \cdots \cdots  &  \alpha_1^{q^{i_{k-1}}}\\
      \vdots\quad & \vdots\quad\qquad  &   \ddots  & \vdots\qquad\quad  \\
        \alpha_{k-1}^{q^{i_0}}  &   \alpha_{k-1}^{q^{i_1}}   & \cdots \cdots  &  \alpha_{k-1}^{q^{i_{k-1}}}
 \end{array}
 \right).\]
The set $I$ is a {\em Moore exponent set} for $q$ and $n$ if
\[\det\,M_{A, I} = 0 \mbox{ if and only if } \alpha_0, \alpha_1 \ldots, \alpha_{k-1}\mbox{ are  } \mathbb{F}_q\mbox{-linearly dependent.}\]
The definition of Moore exponent set was introduced in \cite{DANYU}.

\begin{remark}\label{rem:mooresmall}
If $k>n$, then two elements of $I$ are equal modulo $n$ and hence two columns of $M_{A,I}$ are the same. Thus, a Moore exponent set for $q$ and $n$ has size at most $n$.
\end{remark}

A first  example of Moore exponent set is $I = \{0, 1, \ldots, k-1\}$ for every prime power $q$ and every $n\geq k$; in such a case, $M_{A,I}$ is a square Moore matrix as originally introduced in \cite{Moore}.
For every $s\in\mathbb{N}$, a set $I$ of non-negative integers is a Moore exponent set for $q$ and $n$ if and only if $I+s = \{i+s \pmod{n} : i \in I\}$ is a Moore exponent set for $q$ and $n$.
Thus, we may always assume that the smallest element in $I$ is $0$. Besides $I = \{0, 1, \ldots , k - 1\}$, the following are other known examples of Moore exponent sets.
\begin{itemize}
\item  $I = \{0, d, \ldots , (k - 1)d\}$ for any $q$ and $n$ satisfying $\gcd(d, n) = 1$, that is, $I$ is given by the first elements of an arithmetic progression whose common difference is coprime with $n$; see \cite{kshevetskiy_new_2005}.
\item  $I = \{0, 1, 3\}$ for $n = 7$ and odd $q$; see \cite{CMPZ}.
\item  $I = \{0, 1, 3\}$ for $n = 8$ and $q \equiv 1 \pmod3$; see \cite{CMPZ}.
\item $I =\{0, 2, 3, 4\}$, for $n= 7$ and odd $q$; see \cite{CMPZ}.
\item $I =\{0, 2, 3, 4, 5\}$, for $n=8$ and  $q\equiv 1 \pmod3$; see \cite{CMPZ}.
\end{itemize}


\begin{theorem}\label{th:MESescatt} {\rm \cite[Theorem 2.5]{CMPZ}} Let $I=\{i_0,\ldots,i_{k-1}\}$ be a set of $k$  non--negative integers. The following are equivalent:
\begin{itemize}
\item $I$  is a Moore exponent set for $q$ and $n$;

\item $\cC=\langle x^{q^{i_0}},  x^{q^{i_1}}, \ldots,  x^{q^{i_{k-1}}} \rangle_{\fqn}\subseteq \mathcal{L}_{n,q}$ is an MRD-code;

\item $L_{U_{\cC}} = \{\langle (x^{q^{i_0}},  x^{q^{i_1}}, \ldots,  x^{q^{i_{k-1}}}) \rangle_{\fqn}\; : \, x\in {\mathbb{F}_{q^n}^*}\}$ is a scattered $\fq$-linear set with respect to the hyperplanes of $\PG(k-1,q^n)$.
\end{itemize}
\end{theorem}





Some easy consequences follow from Theorem \ref{th:MESescatt}.
Let $h\geq2$, $I=\{0,i_1,\ldots,i_h\}\subseteq\{0,\ldots,n-1\}$; for $j=1,\ldots,h$, let $\sigma_j:\fqn\to\fqn$ be defined by $\sigma_j(x)= x^{q^{i_j}}$.

\begin{proposition}\label{lemma:fieldintersection}
If $I$ is a Moore exponent set for $q$ and $n$, then $\cap_{j=1}^h \mathrm{Fix}(\sigma_j)=\fq$.
\end{proposition}

\begin{proof}
Let $\mathbb{F}_{q^\ell}=\cap_{j=1}^h \mathrm{Fix}(\sigma_j)$, so that $\ell$ divides $i_j$ for any $j$ and $\ell$ divides $n$.
Let $\mathcal{C}=\langle x,x^{q^{i_1}},\ldots,x^{q^{i_h}}\rangle_{\fqn}$ be the MRD-code associated with $I$. Then every $g\in\mathcal{C}$ is a $q^\ell$-polynomialm and hence $\dim{\fq}\ker g$ is a multiple of $\ell$.
As $\mathcal{C}$ is an MRD-code, by Remark \ref{rem:MRDpol}, there exists $g\in\mathcal{C}$ with $\dim_{\fq}\ker g=1$. 
This implies $\ell=1$, that is the claim.
\end{proof}

If $h=2$, Proposition \ref{lemma:fieldintersection} can be specialized as follows.

\begin{proposition}\label{per2-scattered} Let $I=\{0, i_1, i_2\}$ be a Moore exponent set for $q$ and $n$. Then, either  ${\rm Fix}(\sigma_1) =\mathbb{F}_q$ or ${\rm Fix}(\sigma_2) =\mathbb{F}_q$.
\end{proposition}

\begin{proof} Assume to the contrary  that $\mathrm{Fix}(\sigma_1) =\mathbb{F}_{q^\ell}$ and $\mathrm{Fix}(\sigma_2) =\mathbb{F}_{q^m}$, with $\ell, m > 1$. Being  $\{0, i_1, i_2\}$  a Moore exponent set, it follows that ${\mathcal C} = \langle x, x^{q^{i_1}},  x^{q^{i_2}}\rangle_{\mathbb{F}_{q^n}}$ is an MRD-code. 
Since $\mathrm{Fix}(\sigma_{j}) = \ker(x^{q^{i_j}}-x)$ for $j\in\{1, 2\}$, one has that $\dim_{\mathbb{F}_q}\ker(x^{q^{i_j}}-x)\le 2$ by Remark \ref{rem:MRDpol}; hence $\ell=m=2$ and $n$ is even. It follows that $i_1$ and $i_2$ are even. Therefore, the elements in $\mathcal C$ are $q^2$--polynomials. Thus, we have a contradiction to Remark \ref{rem:MRDpol}. 
\end{proof}

\begin{corollary}
The Gabidulin index of an MRD-code $\cC=\langle x,x^{q^{i_1}},x^{q^{i_2}}\rangle_{\fqn}$ is greater than or equal to $2$. 
In particular, if $\cC$ is not equivalent to any generalized Gabidulin code, its Gabidulin index is two. 
\end{corollary}
\begin{proof}
By Proposition \ref{per2-scattered}, we have either $\gcd(i_1,n)=1$ or $\gcd(i_2,n)=1$, i.e. either $\langle x,x^{q^{i_1}}\rangle_{\fqn}=\mathcal{G}_{2,i_1}$ or $\langle x,x^{q^{i_2}}\rangle_{\fqn}=\mathcal{G}_{2,i_2}$. The claim follows.
\end{proof}

It was shown in \cite{DANYU} that, if $I$ is a Moore exponent set for $q$ and $n$ and $n$ is big enough with respect to the elements of $I$, then $I$ is given by the first elements of an arithmetic progression.

\begin{theorem}{\rm \cite[Theorems 1.1, 3.2 and 4.1]{DANYU}}\label{th:DanyU}
Let $I$ be a Moore exponent set for $q$ and $n$, with $|I|>2$.
Let $j$ be the largest element of $I$ and define $N$ to be either $4j+2$ or $\frac{13}{3}j+2$ according to $|I|=3$ or $|I|>3$, respectively.
If $q>5$ and $n>N$, then $I$ is given by the first elements of an arithmetic progression.
\end{theorem}

\section{Linear sets of $h$-pseudoregulus type}\label{sec:hpseudo}

The following definition generalizes the concept of linear set of pseudoregulus type.

\begin{definition}\label{def:hpseudo}
Let $h\geq1$, $t,n\geq2$.
An $\fq$-linear set $L$ of rank $nt$ in $\Lambda={\rm PG}((h+1)t-1,q^n)$ such that $\langle L \rangle=\Lambda$ is of $h$-pseudoregulus type if:
\begin{itemize}
\item[(a)] there exist $s=\frac{q^{nt}-1}{q^n-1}$ pairwise disjoint $h$-subspaces $\pi_1,\ldots,\pi_s$ of $\Lambda$ such that $w_L(\pi_j)=n$ for any $j\in\{1,\ldots,s\}$;
\item[(b)] there exist exactly $h+1$ distinct $(t-1)$-subspaces $T_1,\ldots,T_{h+1}$ of $\Lambda$ such that $\Lambda=\langle T_1,\ldots,T_{h+1}\rangle$, $T_i\cap \pi_j\ne\emptyset$ for any $i$ and $j$, and $L\cap K_{i_0}=\emptyset$ for any $i_0\in\{1,\ldots,h+1\}$, where $K_{i_0}=\langle \cup_{i\ne i_0} T_{i}\rangle$.
\end{itemize}
The set $\mathcal{P}=\{\pi_1,\ldots,\pi_s\}$ is called the $h$-pseudoregulus associated with $L$, and $T_1,\ldots,T_{h+1}$ are the transversal spaces of $\mathcal{P}$.
\end{definition}

Note that from $\langle L \rangle=\Lambda$ follows $n\geq h+1$. 

We start by showing examples of linear sets of $h$-pseudoregulus type; we will then prove that every linear set of $h$-pseudoregulus type is equivalent to one of such examples.

\begin{theorem}\label{th:fromalgebratogeometry}
Let $\Lambda={\rm PG}(\mathbb{F}_{q^{nt}}^{h+1},\mathbb{F}_{q^n})=\PG((h+1)t-1,q^n)$, where $h,t,n$ are positive integers with $t,n\geq2$ and $n\geq h+1$.
For $j=2,\ldots,h+1$, let $f_j:\mathbb{F}_{q^{nt}}\to \mathbb{F}_{q^{nt}}$ be an invertible strictly $\fqn$-semilinear map with companion automorphism $\sigma_j\in\mathrm{Gal}(\mathbb{F}_{q^n}\colon\mathbb{F}_q)$ such that $\sigma_2,\ldots,\sigma_{h+1}$ are pairwise distinct.
Denote by $\underline{f}$ the $(h+1)$-tuple $(f_1=\mathrm{id},f_2,\ldots,f_{h+1})$.
Let
\begin{equation}\label{eq:pseudo}
U_{\underline{f}}=\{(x,f_2(x),\ldots,f_{h+1}(x))\mid x\in\mathbb{F}_{q^{nt}}\},
\end{equation}
\begin{equation}\label{eq:linearpseudo}
L_{\underline{f}}:=L_{U_{\underline{f}}}=\{\langle(x,f_2(x),\ldots,f_{h+1}(x))\rangle_{\fqn}\mid x\in\mathbb{F}_{q^{nt}}^*\} \subset \Lambda.
\end{equation}
Then $L_\ff$ is an $\fq$-linear set of $\Lambda$ of $h$-pseudoregulus type.
Moreover, for any $x\in\fqnt^*$, define
\[W_x=\langle(x,0,\ldots,0),(0,f_2(x),\ldots,0),\ldots,(0,\ldots,0,f_{h+1}(x))\rangle_{\fqn},\]
\[ \Pi_x={\rm PG}(W_x,\fqn)={\rm PG}(h,q^n).\]
For $i=1,\ldots,h+1$, let
\begin{equation}\label{eq:transversal}
T_i=\PG(\langle\mathbf{e}_i\rangle_{\fqnt},\fqn)=\PG(t-1,q^n)\subset\Lambda,
\end{equation}
where $\mathbf{e}_i$ is the unit $i$-th vector of the standard ordered basis of $\fqnt^{h+1}$.
Then $\mathcal{P}=\{\Pi_x\mid x\in\fqnt^*\}$ is the $h$-pseudoregulus of $L_\ff$ and $T_1,\ldots,T_{h+1}$ are the transversal spaces of $\mathcal{P}$.
\end{theorem}

First we show the following lemma.

\begin{lemma}\label{lemma:span}
Let $x_1,\ldots, x_s \in \fqnt$. Then $x_1,\ldots, x_s$ are $\fqn$-independent if and only if $\dim \langle \Pi_{x_1},\ldots,\Pi_{x_s}\rangle=(h+1)s-1$.
Also, if $\Pi\in\mathcal{P}$ satisfies $\Pi \cap\langle\Pi_{x_1},\ldots,\Pi_{x_s}\rangle\ne\emptyset$, then $\Pi \subseteq \langle\Pi_{x_1},\ldots,\Pi_{x_s}\rangle$.
\end{lemma}
\begin{proof}
Suppose $\dim \langle \Pi_{x_1},\ldots,\Pi_{x_s}\rangle<(h+1)s-1$.
Let $i \in \{1,\ldots,s\}$ such that there exists $P \in \Pi_{x_i}\cap \langle \Pi_{x_j} \colon j\ne i \rangle$.
Write $P=\langle(\lambda_1 x_i,\ldots,\lambda_{\ell}f_{\ell}(x_i),\ldots,\lambda_{h+1}f_{h+1}(x_i))\rangle_{\fqn}$.
For some $\ell\in\{1,\ldots,h+1\}$, the $\ell$-th component of a vector defining $P$ is non-zero, i.e. $\lambda_\ell f_\ell (x_i)\ne0$ and $\lambda_\ell f_\ell(x_i)=\sum_{j\ne i}\mu_{j}f_\ell(x_j)$ with $\mu_j\in\fqn$.
This implies that $x_1,\ldots,x_s$ are $\fqn$-linearly dependent.
Conversely, suppose that $x_1,\ldots,x_s$ are $\fqn$-linearly dependent. Then there exists $i \in \{1,\ldots,s\}$ such that $x_i$ is an $\fqn$-linear combination of $\{x_j \colon j\ne i\}$. Thus, easy computations show that $\Pi_{x_i}\subseteq \langle \Pi_{x_j} \colon j \neq i \rangle$, and hence $\dim \langle \Pi_{x_1},\ldots,\Pi_{x_s}\rangle<(h+1)s-1$.
This also proves the last part of the claim.
\end{proof}

\noindent \emph{Proof of Theorem \ref{th:fromalgebratogeometry}.} 
Let $x\in\fqnt^*$; we determine $\Pi_x \cap L_\ff$.
Let $P=\langle (y,f_2(y),\ldots,f_{h+1}(y))\rangle_{\fqn}\in L_\ff\cap \Pi_x$ for some $y\in\fqnt^*$. Then there exist $\lambda_1,\ldots,\lambda_{h+1}\in\fqn$ such that
$$ (\lambda_1 x,\lambda_2 f_2(x),\ldots,\lambda_{h+1} f_{h+1}(x)) =(y,f_2(y),\ldots,f_{h+1}(y)).$$
Since $f_2,\ldots,f_{h+1}$ are invertible and $\fqn$-semilinear maps with companion automorphisms $\sigma_2,\ldots,\sigma_{h+1}$ respectively, this is equivalent to
$$ \left\{ \begin{array}{crcrcrcr}
y&=&\lambda_1 x, \\
\lambda_2&=&\lambda_1^{\sigma_2}, \\
&\vdots& \\
\lambda_{h+1}&=&\lambda_1^{\sigma_{h+1}}. \\
\end{array}\right.
 $$
Thus, $L_{\ff}\cap\Pi_x=\{\langle(\lambda_1 x,\lambda_1^{\sigma_2}f_2(x),\ldots,\lambda_1^{\sigma_{h+1}}f_{h+1}(x))\rangle_{\fqn}\mid\lambda_1\in\fqn^*\}$.
Let 
\[ W=\langle\mathbf{e}_1\rangle_{\fqn}+\ldots+ \langle\mathbf{e}_{h+1}\rangle_{\fqn}=\fqn^{h+1} \subset \fqnt^{h+1}.\]
Clearly, the map $F_x\colon W_x \rightarrow W$ defined by
\[ F_x(\lambda_1x,\ldots,\lambda_{h+1} f_{h+1}(x))=(\lambda_1,\ldots,\lambda_{h+1}), \]
is an invertible $\fqn$-linear map such that
\begin{equation}\label{eq:standard}
F_x(U_\ff\cap W_x)=\{(\lambda,\lambda^{\sigma_2},\ldots,\lambda^{\sigma_{h+1}})\mid\lambda\in\fqn^*\};
\end{equation}
hence $L_\ff\cap\Pi_x$ and the linear set $\{\langle(\lambda,\lambda^{\sigma_2},\ldots,\lambda^{\sigma_{h+1}})\rangle_{\fqn}\mid\lambda\in\fqn^*\}$ are ${\rm PGL}(h+1,q^n)$-equivalent, and the weight of $\Pi_x$ with respect to $L_\ff$ is $n$.
By Lemma \ref{lemma:span}, we get {\it (a)} of Definition \ref{def:hpseudo}.
Also, since $\sigma_2,\ldots,\sigma_{h+1}$ are pairwise distinct, $\langle L\cap \Pi_x \rangle=\Pi_x$ for every $x\in\fqnt^*$ and hence $\langle L \rangle=\Lambda$.
Clearly, $\Lambda=\langle T_1,\ldots,T_{h+1}\rangle$ and $T_i\cap\Pi_x\ne\emptyset$ for each $i=1,\ldots,h+1$ and $x\in\fqnt^*$.

Furthermore, $L_{\ff}\cap K_{i_0}=\emptyset$ for each $i_0\in\{1,\ldots,h+1\}$. Indeed, if $L_{\ff}\cap K_{i_0}\ne\emptyset$, then there exists $y\in\fqnt^*$ such that $f_{i_0}(y)=0$, a contradiction.

Now we prove that $T_1,\ldots,T_{h+1}$ are the unique transversal spaces of $\mathcal{P}$.
Let $T$ be a transversal space; since by Definition \ref{def:hpseudo} $\mathcal{P}$ has exactly $\frac{q^{nt}-1}{q^n-1}$ pairwise disjoint elements and $|\mathcal{P}|=|T|$, we have that $T$ intersects every $\pi\in\mathcal{P}$ in exactly one point.
As $\cap_{i=1}^{h+1}K_i=\emptyset$, there exists $\ell\in\{1,\ldots,h+1\}$ such that $T\not\subseteq K_\ell$; we show that $T=T_\ell$.
Let $P_1,\ldots,P_t\in T\setminus K_{\ell}$ be such that $T=\langle P_1,\ldots,P_t\rangle$. For any $i$, let $P_i \in \Pi_{x_i}$.
For each $\Pi\in\mathcal{P}$, we have $\Pi\cap T\ne\emptyset$; since $T\subseteq\langle\Pi_{x_1},\ldots, \Pi_{x_t}\rangle$, Lemma \ref{lemma:span} yields $\Pi\subseteq\langle\Pi_{x_1},\ldots, \Pi_{x_t}\rangle$.
Thus, $\dim \langle \Pi_{x_1},\ldots,\Pi_{x_t}\rangle =(h+1)t-1$ and hence, by Lemma \ref{lemma:span}, $x_1,\ldots,x_t$ are $\fqn$-linearly independent.
For any $i=1,\ldots,t$, write $P_i=\langle(\lambda_{x_i}^{(1)} x_i,\ldots,\lambda_{x_i}^{(\ell)}f_\ell(x_i),\ldots,\lambda_{x_i}^{(h+1)}f_{h+1}(x_i))\rangle_{\fqn}$ with $\lambda_{x_i}^{(j)}\in\fqn$. Since $P_i\notin K_\ell$, we have $\lambda_{x_i}^{(\ell)}f_{\ell}(x_i)\ne0$.
As $x_1,\ldots,x_t$ and hence $f_\ell(x_1),\ldots,f_\ell(x_t)$ are $\fqn$-linearly independent, we have $T\cap K_\ell=\emptyset$.
Therefore, since $T\cap\Pi_{x}\ne\emptyset$ for every $x\in\fqnt^*$, $T=\PG(\tilde{W},\fqn)$ with
\[\tilde{W}=\{(\lambda_x^{(1)} x,\lambda_x^{(2)}f_2(x),\ldots,\lambda_x^{(h+1)}f_{h+1}(x))\colon x\in\fqnt\},\]
where $\lambda_x^{(j)}\in\fqn$ and $\lambda_x^{(\ell)}\ne0$ for every $x \in \fqnt^*$.
Let $i \in \{1,\ldots,h+1\}$.
Since $\tilde{W}$ is an $\fqn$-vector space and $f_i$ is invertible and $\fqn$-semilinear with associated $\fq$-automorphism $\sigma_i$, we have, for all $x,y\in\fqnt$,
\[\lambda_x^{(i)}f_i(x)+\lambda_y^{(i)}f_i(y)=\lambda_{x+y}^{(i)}f_i(x+y),\]
and hence
\[\left((\lambda_x^{(i)})^{\sigma_i^{-1}}-(\lambda_{x+y}^{(i)})^{\sigma_i^{-1}}\right)x=\left((\lambda_y^{(i)})^{\sigma_i^{-1}}-(\lambda_{x+y}^{(i)})^{\sigma_i^{-1}}\right)y.\]
If $x$ and $y$ are $\fqn$-linearly independent, this yields $\lambda_x^{(i)}=\lambda_{x+y}^{(i)}=\lambda_y^{(i)}$; if $x$ and $y$ are $\fqn$-linearly dependent and $z\in \fqnt\setminus\langle x,y\rangle_{\fqn}$, then $\lambda_x^{(i)}=\lambda_{z}^{(i)}=\lambda_y^{(i)}$.
Thus, $\lambda_x^{(i)}=\lambda^{(i)}$ is constant for all $x\in\fqnt$.
Therefore,
\[\tilde{W}=\{(\beta_1 x,\ldots, f_\ell(x),\ldots,\beta_{h+1}g_{h+1}(x))\colon x\in\fqnt\},\]
where $\beta_i=\lambda^{(i)}/\lambda^{(\ell)}$ for every $i \in \{1,\ldots,h+1\}$.
Since $\tilde{W}$ is closed under scalar multiplication in $\fqn$ and $f_i$ is strictly $\fqn$-semilinear for every $i\ne1$, this implies that $\beta_i=0$ for $i\neq \ell$, i.e. $T=T_{\ell}$.
\qed

\begin{remark}
It is readily seen from  \eqref{eq:standard} that the slices $L_\ff \cap\Pi_x$ of the linear set $L_\ff$ lying on each space $\Pi_x$ are all ${\rm PGL}\left((h+1)t,q^n\right)$-equivalent.
\end{remark}

\begin{remark}\label{rem:equiv}
Let $V=V((h+1)t,q^n)=S_1\oplus\ldots\oplus S_{h+1}$ where $\dim_{\fqn}S_i=t$.
For $i=2,\ldots,h+1$, let $g_i:S_1\to S_i$ be an invertible strictly $\fqn$-semilinear map with companion automorphism $\sigma_i\in\mathrm{Gal}(\fqn:\fq)$, such that $\sigma_2,\ldots,\sigma_{h+1}$ are pairwise distinct.
Define
\[\tilde{U}_{\underline{g}}=\{v+g_2(v)+\ldots+g_{h+1}(v) \mid v \in S_1\}\]
and let $\varphi: V\to\fqnt^{h+1}$ be an $\fqn$-linear isomorphism such that $\varphi(S_i)=\langle\mathbf{e}_i\rangle_{\fqnt}$.
Then
\begin{equation}\label{eq:pseudogen}
\varphi(\tilde{U}_{\underline{g}})=U_\ff
\end{equation}
where $f_i:\fqnt\to\fqnt$ is the $i$-th component of $\varphi\circ g_i\circ\varphi^{-1} (x,0,\ldots,0)$, and $f_i$ has companion automorphism $\sigma_i$.

In the following we will say that $\tilde{U}_{\underline{g}}$ and $U_\ff$ are equivalent, because of  \eqref{eq:pseudogen}. 
\end{remark}

\subsection{Analytic shape of a linear set of $h$-pseudoregulus type}

We now want to prove a converse of Theorem \ref{th:fromalgebratogeometry}; that is, we show in the next result that every linear set of $h$-pseudoregulus type is equivalent to a linear set $L_\ff$.
More precisely,

\begin{theorem}\label{thm:maintheoremsec3}
Let $L$ be an $\fq$-linear set of $h$-pseudoregulus type in $\Lambda=\PG(V,\fqn)=\PG((h+1)t-1,q^n)$ and let $T_i=\PG(W_i,\fqn)$ be its transversal spaces for $i\in\{1,\ldots,h+1\}$.
Then there exist invertible strictly $\fqn$-semilinear maps $f_i:W_1\to W_i$, $i=2,\ldots,h+1$, with pairwise distinct companion $\fq$-automorphisms, such that
\[ L=\{\langle w+f_2(w)+\ldots+f_{h+1}(w) \rangle_{\fqn} \colon w \in W_1^*\}, \]
hence $L$ is equivalent to a linear set of Form \eqref{eq:linearpseudo}.
\end{theorem}

In order to prove Theorem \ref{thm:maintheoremsec3}, we need to look at the considered linear set as projection of a subgeometry. 

\begin{theorem}\label{thm:checceserve}
Let $\Sigma=\PG(nt-1,q)$ be a subgeometry of $\Sigma^*=\PG(S^*,\fqn)=\PG(tn-1,q^n)$, $\Psi\in{\rm P\Gamma L}(\Sigma^*)$ be a collineation of order $n$ such that $\mathrm{Fix}(\Psi)=\Sigma$.
Let $i_1,\ldots,i_{h+1}$ be distinct elements of $\{0,\ldots,n-1\}$.
Let $\mathcal{D}$ be a Desarguesian $(n-1)$-spread of $\Sigma$ with director subspace $\Theta$.
Define
\begin{itemize}
    \item $\Gamma=\langle\Theta^{\Psi^i}\mid i\in \{0,\ldots,n-1\}\setminus\{i_1,\ldots,i_{h+1}\}\rangle$,
    \item $\Lambda=\langle\Theta^{\Psi^i}\mid i\in \{i_1,\ldots,i_{h+1}\}\rangle$.
\end{itemize}
Let $L$ be the projection of $\Sigma$ from $\Gamma$ to $\Lambda$.
Then $L$ is an $\fq$-linear set of $h$-pseudoregulus type in $\Lambda$ with transversal spaces $\Theta^{\Psi^{i_1}},\ldots,\Theta^{\Psi^{i_{h+1}}}$ and $L$ is equivalent to $L_\ff$ as in \eqref{eq:linearpseudo}.
\end{theorem}

\begin{proof}
As $\mathcal{D}$ is a Desarguesian $(n-1)$-spread of $\Sigma$ with director space $\Theta$, we have that $\mathrm{dim}\,\Theta=t-1$, the director spaces of $\mathcal{D}$ are $\Theta,\Theta^{\Psi},\ldots,\Theta^{\Psi^{n-1}}$, and $\Sigma^*=\langle\Theta,\Theta^{\Psi},\ldots,\Theta^{\Psi^{n-1}}\rangle$; see Subsection \ref{sec:desspreads}.
It follows that $\dim\,\Gamma=(n-h-1)t-1$ and $\dim\,\Lambda=(h+1)t-1$. Since $\Sigma^*=\langle\Gamma,\Lambda\rangle$, this implies $\Gamma\cap\Lambda=\emptyset$. Also, $\Gamma\cap\Sigma=\emptyset$ because the $\Theta^{\Psi^j}$'s are pairwise disjoint and cyclically permuted by $\Psi$.
Therefore by Theorem \ref{th:projection} the projection $L=p_{\Gamma,\Lambda}(\Sigma)$ is an $\fq$-linear set of rank $nt$ in $\Lambda$ such that $\Lambda=\langle L\rangle$.

Let $W$ be the $\fqn$-subspace of $S^*$ such that $\Theta=\PG(W,\fqn)$. Denoting by $g\in{\rm \Gamma L}(S^*)$ the semilinear map associated with $\Psi$, we have
\begin{equation}\label{eq:sigma}
\Sigma=\left\{\langle u+g(u)+\cdots+g^{n-1}(u)\rangle_{\fqn}\mid  u\in W^*\right\}.
\end{equation}
In fact, the right-hand side of \eqref{eq:sigma} is fixed by $\Psi$ elementwise and hence is contained in $\Sigma$; also, the dimension over $\fq$ of both right- and left-hand side is equal to $nt-1$, because $\dim_{\fqn}W=t$.
Thus,
\begin{equation}\label{eq:Lproj}
L=p_{\Gamma,\Lambda}(\Sigma)=\left\{\langle g^{i_1}(u)+\cdots+g^{i_{h+1}}(u)\rangle_{\fqn}\mid u\in 
W^*\right\}
\end{equation}
$$ =\left\{\langle v+g^{i_2-i_1}(v)+\cdots+g^{i_{h+1}-i_1}(v)\rangle_{\fqn}\mid v\in g^{i_1}(W)^*\right\}.$$
By Remark \ref{rem:equiv}, $L$ is equivalent to $L_\ff$, where $f_j$ is the $j$-th component of $\varphi\circ g^{i_j-i_1}\circ\varphi^{-1}(x,0\ldots,0)$ and $S_j=g^{i_j}(W)$.
Denoting by $\sigma$ the companion automorphism of $g$, the automorphism associated with $f_j$ is $\sigma^{i_j-i_1}$, so that $f_j$ is not $\fqn$-linear.
\end{proof}

\begin{theorem}\label{th:fromgeometrytoalgebra}
Let $L$ be an $\fq$-linear set of rank $nt$ in $\Lambda=\PG((h+1)t-1,q^n)$ of $h$-pseudoregulus type with $h$-presudoregulus $\mathcal{P}$, having transversal spaces $T_1,\ldots,T_{h+1}$.
Let $\Lambda$ be embedded in $\Sigma^*=\PG(tn-1,q^n)$.
Let $\Sigma=\PG(tn-1,q)$ be a subgeometry of $\Sigma^*$, $\Psi\in{\rm P\Gamma L}(\Sigma^*)$ be such that $\mathrm{Fix}(\Psi)=\Sigma$, and $\Gamma=\PG((n-h-1)t-1,q^n)\subseteq \Sigma^*$ be such that $L$ is the projection $p_{\Gamma,\Lambda}(\Sigma)$ from $\Gamma$ to $\Lambda$ of $\Sigma$.
Then
\begin{itemize}
\item the set $\mathcal{D}_L=\{\langle\Gamma,\pi\rangle\cap\Sigma\mid\pi\in\mathcal{P}\}$ is a Desarguesian $(n-1)$-spread of $\Sigma$;
\item there exists a subset $\{\ell_1:=0,\ell_2,\ldots,\ell_{h+1}\}$ of $\{0,1,\ldots,n-1\}$, and a director space $\overline{\Theta}$ of $\mathcal{D}_L$, such that $\Gamma=\langle\overline{\Theta}^{\Psi^j}\mid j\notin\{0,\ell_2,\ldots,\ell_{h+1}\}\rangle$;
\item $T_i=\langle \Gamma, \overline{\Theta}^{\Psi^{\ell_i}} \rangle \cap \Lambda$ for $i \in \{1,\ldots,h+1\}$.
\end{itemize}
\end{theorem}

\begin{proof}
For $\pi\in\mathcal{P}$, let $X_{\pi}=\langle\Gamma,\pi\rangle\cap\Sigma$.
Since $L\cap\pi$ is an $\fq$-linear set in $\pi$ of rank $n$, $X_\pi$ is an $(n-1)$-subspace of $\Sigma$. Also, by Definition \ref{def:hpseudo}(a), $\mathcal{D}_L$ is an $(n-1)$-spread of $\Sigma$. We show that $\mathcal{D}_L$ is Desarguesian.

Let $T_1,\ldots,T_{h+1}$ be the transversal spaces of $\mathcal{P}$. For any $i=1,\ldots,h+1$, define the subspaces $K_i=\langle\cup_{j\ne i} T_j\rangle$ and $\overline{K}_i=\langle\Gamma,K_i\rangle$ of $\Sigma^*$.
Clearly, $\dim\, \overline{K}_i=(n-1)t-1$.
Also, $\overline{K}_i\cap\Sigma=\emptyset$; in fact, if $Q\in \overline{K}_i\cap\Sigma$, then $\langle \Gamma,Q\rangle\cap\Lambda\in p_{\Gamma,\Lambda}(\Sigma)=L$ and $\langle \Gamma,Q\rangle\cap\Lambda\in \overline{K}_i\cap\Lambda=\langle\cup_{j\ne i}T_j\rangle$, a contradiction to $L\cap\langle\cup_{j\ne i}T_j\rangle=\emptyset$ ($(b)$ of Definiton \ref{def:hpseudo}).
This implies that the $\Psi$-invariant subspace $\overline{K}_i\cap \overline{K}_i^{\Psi}\cap\cdots\cap \overline{K}_i^{\Psi^{n-1}}$ of $\overline{K}_i$ is empty.

Let $i=1$ and $\Theta=\overline{K}_1\cap \overline{K}_1^{\Psi}\cap\cdots\cap \overline{K}_1^{\Psi^{n-2}}$.
An easy induction on $n$ proves that $\dim\, \Theta \geq t-1$. Together with $\Theta\cap \overline{K}_1^{\Psi^{n-1}}=\emptyset$, this yields $\dim\, \Theta=t-1$ and $\Sigma^*=\langle \Theta,\Theta^{\Psi},\ldots,\Theta^{\Psi^{n-1}}\rangle$.

For any $\pi\in\mathcal{P}$, let $X_\pi^*=\langle X_\pi\rangle$ be the $(n-1)$-subspace of $\Sigma^*$ such that $X_\pi^*\cap\Sigma=X_\pi$; then $X_\pi^*\subseteq\langle\Gamma,\pi\rangle$. Since $X_\pi^*\cap\Sigma=\PG(n-1,q)$ and $\dim\, X_\pi^*=n-1$, Lemma \ref{lemma:lunardon} implies that $X_\pi^*$ is $\Psi$-invariant.
Let $r=\pi\cap\overline{K}_1$.
Since $\dim\, r=h-1$, we have that $\langle\Gamma,r\rangle$ is a hyperplane of $\langle\Gamma,\pi\rangle$.
Also, $X_\pi^*\not\subseteq\langle\Gamma,r\rangle$; otherwise, the non-empty subset $p_{\Gamma,\Lambda}(X_\pi)$ of $L$ would be contained in $\langle\Gamma,r\rangle\cap\Lambda=r\subset\langle\cup_{j\ne1}T_j\rangle$, a contradiction by $(b)$ of Definition \ref{def:hpseudo}.
Thus, $H_\pi=X_\pi^*\cap\langle\Gamma,r\rangle$ is a hyperplane of $X_\pi^*$, i.e. $\dim\, H_\pi=n-2$.
Since $H_\pi\subset \overline{K}_1$, we have $H_\pi\cap\Sigma=\emptyset$ and hence $H_\pi\cap H_\pi^{\Psi}\cap\cdots\cap H_\pi^{\Psi^{n-1}}=\emptyset$.
So one gets that $H_\pi\cap H_\pi^{\Psi}\cap\cdots\cap H_\pi^{\Psi^{n-2}}$ is a point $R_\pi$ of $X_\pi^*$.
Also, $R_\pi\in\Theta$ as $H_\pi\subset \overline{K}_1$. Therefore $R_\pi\in X_\pi^*\cap\Theta$, so that $X_\pi^*\cap\Theta$ is non-empty for any $\pi\in\mathcal{P}$. This implies that $\mathcal{D}_L$ is a Desarguesian spread of $\Sigma$ with director spaces $\Theta,\Theta^{\Psi},\ldots,\Theta^{\Psi^{n-1}}$, see Proposition \ref{prop:desspreads}.

By direct computations, $\overline{K}_1=\langle\Theta^{\Psi^{j}}\mid j\ne1\rangle$ and $\overline{K}_1\cap\Theta^{\Psi}=\emptyset$.
Arguing in the same way we obtain for any $i=1,\ldots,h+1$ that $\overline{K}_i\cap \overline{K}_i^{\Psi}\cap\cdots\cap \overline{K}_i^{\Psi^{n-2}}$ is a director space for $\mathcal{D}_L$, hence there exists $i \in \{1,\ldots,h+1\}$ such that $\overline{K}_i\cap \overline{K}_i^{\Psi}\cap\cdots\cap \overline{K}_i^{\Psi^{n-2}} =\Theta^{\Psi^{\ell_i}}$, $\overline{K}_i=\langle\Theta^{\Psi^j}\mid j\ne \ell_i+1\rangle$, and $\overline{K}_i\cap\Theta^{\Psi^{\ell_i+1}}=\emptyset$; here, $\ell_1=0,\ell_2,\ldots,\ell_{h+1}$ are distinct integers in $\{0,\ldots,n-1\}$.
Now we have
$$ \Gamma=\bigcap_{i=1}^{h+1}\overline{K}_i =\big\langle\Theta^{\Psi^j}\mid j\notin\{\ell_1+1,\ldots,\ell_{h+1}+1\}\big\rangle. $$
Finally, for every $k\in\{1,\ldots,h+1\}$ we have $\displaystyle T_k=\bigcap_{i\ne k} K_i$ and  \[\langle\Gamma,T_k\rangle\subseteq\bigcap_{i\ne k}\overline{K}_i=\bigcap_{i\ne k}\langle \Theta^{\Psi^j}\colon j\ne \ell_i+1\rangle=\langle\Gamma,\Theta^{\Psi^{\ell_k+1}}\rangle;\]
by comparing the dimensions, we get $\langle\Gamma,T_k\rangle=\langle\Gamma,\Theta^{\ell_k+1}\rangle$.
It follows that $T_k=\langle \Gamma, \Theta^{\Psi^{\ell_k+1}} \rangle \cap \Lambda$.
Writing $\overline{\Theta}=\Theta^{\Psi}$, the claim follows.
\end{proof}


\noindent {\it Proof of Theorem {\rm \ref{thm:maintheoremsec3}.}}
By using the notations and the claims of Theorem \ref{th:fromgeometrytoalgebra}, there exist $h+1$ distinct non-negative integers $\ell_1=0,\ell_2,\ldots,\ell_{h+1}$ and a director space $\overline{\Theta}=\PG(\overline{W},\fqn)$ of $\mathcal{D}_L$ such that $T_i=\langle\Gamma,\overline{\Theta}^{\Psi^{\ell_i}}\rangle\cap\Lambda$ for every $i$.
Define $\overline{\Lambda}=\langle\overline{\Theta}^{\Psi^{\ell_j}}\colon j\in\{1,\ldots,h+1\}\rangle=\PG(\overline{V},\fqn)$.
By Theorem \ref{thm:checceserve}, the $\fq$-linear set $\overline{L}=p_{\Gamma,\overline{\Lambda}}(\Sigma)$ can be written as follows
\[\overline{L}=\left\{\langle u+g^{\ell_2}(u)+\cdots+g^{\ell_{h+1}}(u)\rangle_{\fqn}\mid u\in 
\overline{W}^*\right\},\]
where $g\in \Gamma\mathrm{L}(S^*)$ is the semilinear map associated with $\Psi$; see \eqref{eq:Lproj}. 
By Proposition \ref{prop:axisimmaterial}, there exists an $\fqn$-linear isomorphism $\omega:V\to\overline{V}$ such that $\omega(U)=\overline{U}$, where $L=L_U$ and $\overline{L}=L_{\overline{U}}$.
The transversal spaces of the $h$-pseudoregulus $\overline{\mathcal{P}}$ of $\overline{L}$ are defined by $\overline{W}_i=g^{\ell_i}(\overline{W})$, $i=1,\ldots,h+1$.
We have $\omega(W_i)=\overline{W}_i$, and hence the $\fqn$-semilinear invertible map $f_i=\omega^{-1}\circ g^{\ell_i}\circ\omega$ satisfies  $f_i(W_1)=W_i$.

Therefore, \[L=\varphi_{\omega}^{-1}(\overline{L})=\{\langle w+f_2(w)+\ldots+f_{h+1}(w) \rangle_{\fqn} \colon w \in W_1^*\}.\]
\qed

\section{Maximum $h$-scattered linear sets of $h$-pseudoregulus type}\label{sec:maxhscatt}

In this section we characterize maximum $h$-scattered linear sets of $h$-pseudoregulus type by means of Moore exponent sets of size $h+1$.
For $h=1$, maximum $1$-scattered linear set of $1$-pseudoregulus type are exactly the classical maximum scattered linear sets of pseudoregulus type in $\PG(2t-1,q^n)$, which have been characterized in \cite[Theorem 3.5]{LMPT:14}.
In this case, the only automorphism involved has the shape $x\in \fqn\mapsto x^{q^s}\in \fqn$ with $\gcd(s,n)=1$.


Let $h\geq1$ and $t,n\geq2$. For $j=2,\ldots,h+1$, let $f_j:\fqnt\to\fqnt$ be an invertible $\fqn$-semilinear map with companion automorphism $\sigma_j\in\mathrm{Gal}(\fqn:\fq)$, $\sigma_j:x\mapsto x^{q^{i_j}}$, where $i_j\in\{0,\ldots,n-1\}$; also, assume that the automorphisms $\sigma_2,\ldots,\sigma_{h+1}$ are pairwise distinct. 
Let $I=\{i_1=0,i_2,\ldots,i_{h+1}\}$.
Define as in Section \ref{sec:hpseudo} the following objects: $\Lambda$, $\ff$, $U_\ff$, $L_\ff$, $\Pi_x$ for any $x\in\fqnt^*$, $T_1,\ldots,T_{h+1}$.

We start this section by characterizing, among the linear sets of $h$-pseudoregulus type, the maximum $h$-scattered ones.
Recall that, by Remark \ref{rem:equiv} and Theorem \ref{thm:maintheoremsec3}, every linear set of $h$-pseudoregulus type is equivalent to some $L_\ff$ as in \eqref{eq:linearpseudo}.

\begin{theorem}\label{th:mainsec4}
Let $L$ be an $\fq$-linear set of $h$-pseudoregulus type in $\PG((h+1)t-1,q^n)$, where $L$ is equivalent to $L_\ff$ as in \eqref{eq:linearpseudo}. For every $j=2,\ldots,h+1$, let $i_j\in\{1,\ldots,n-1\}$ be such that $x\mapsto x^{q^{i_j}}$ is the companion automorphism of $f_{j}$; let $I=\{0,i_2,\ldots,i_{h+1}\}$.
Then $L$ is maximum $h$-scattered if and only if $I$ is a Moore exponent set for $q$ and $n$.
\end{theorem}
\begin{proof}
Assume that $I$ is a Moore exponent set for $q$ and $n$.
Let $x_1,\ldots,x_t$ be an $\fqn$-basis of $\fqnt$. Then by direct computations $\Lambda=\PG(W_{x_1}\oplus\ldots\oplus W_{x_t},\fqn)$ and
\begin{equation}\label{eq:uff}U_\ff=(U_\ff\cap W_{x_1})\oplus\ldots\oplus(U_\ff\cap W_{x_t}).
\end{equation}
By \eqref{eq:standard}, $U_\ff\cap W_{x_i}$ is $\mathrm{GL}(h+1,q^n)$-equivalent to $\{(\lambda,\lambda^{\sigma_2},\ldots,\lambda^{\sigma_{h+1}})\mid\lambda\in\fqn\}$. The latter defines a $h$-scattered $\fq$-linear set in $\PG(h,q^n)$ because of Theorem \ref{th:MESescatt}; thus,  $L_\ff\cap\Pi_{x_i}$ is a $h$-scattered $\fq$-linear set in $\Pi_{x_i}$.
From Theorem \ref{th:directsum} and \eqref{eq:uff} follows that $L_\ff$ is a $h$-scattered $\fq$-linear set in $\Lambda$.
As $L_\ff$ has rank $nt$, Theorem \ref{th:bound} shows that $L_\ff$ is a maximum $h$-scattered linear set.
The remaining parts of the claim follow from Theorem \ref{th:fromalgebratogeometry}.
Viceversa, assume that $L=L_\ff$ is a maximum $h$-scattered $\fq$-linear set of $h$-pseudoregulus type in $\Lambda=\PG((h+1)t-1,q^n)$.
Since $L$ is $h$-scattered, the linear set obtained as the intersection $L\cap\Pi$ of $L$ with an element $\Pi=\PG(h,q^n)$ of the $h$-pseudoregulus of $L$ is a $h$-scattered $\fq$-linear set of rank $n$ of $\Pi$ equivalent to 
\[\{\langle(\lambda,\lambda^{\sigma_2},\ldots,\lambda^{\sigma_{h+1}})\rangle_{\fqn}\mid\lambda\in\fqn^*\}.\] 
By Theorem \ref{th:MESescatt}, this proves that $I$ is a Moore exponent set for $q$ and $n$.
\end{proof}





We show that every maximum $h$-scattered linear set in $\Lambda$ of $h$-pseudoregulus type, other than the subgeometries of $\Lambda$, has a unique $h$-pseudoregulus.


\begin{theorem}\label{th:uniqPseudo}
Let $L$ be a maximum $h$-scattered $\fq$-linear set in $\Lambda=\PG((h+1)t-1,q^n)$ of $h$-pseudoregulus type, with $h$-pseudoregulus $\mathcal{P}$. Then one of the following holds.
\begin{itemize}
    \item $n=h+1$, $L$ is a subgeometry $\PG((h+1)t-1,q)$ of $\Lambda$; every Desarguesian $h$-spread of $L$ is a $h$-pseudoregulus associated with $L$, whose transversal spaces are the director spaces of the spread.
    \item $n>h+1$, then $w_L(\pi)\leq \frac{hn}{h+1}+1$ for any $h$-subspace $\pi$ of $\Lambda$ with $\pi\notin\mathcal{P}$; in particuar, $\mathcal{P}$ is the unique $h$-pseudoregulus associated with $L$.
\end{itemize}
\end{theorem}

\begin{proof}
Since $L$ is of $h$-pseudoregulus type we have $n\geq h+1$.

Suppose $n=h+1$. Then the rank of $L$ equals $\dim\,\Lambda+1$, and hence $L$ is a subgeometry $\PG((h+1)t-1,q)$ of $\Lambda=\PG((h+1)t-1,q^{h+1})$. 
Let $\mathcal{D}$ be a Desarguesian $h$-spread of $L$. From the properties of Desarguesian spreads (see Subsection \ref{sec:desspreads}) follows immediately that $\mathcal{D}$ is a $h$-pseudoregulus for $L$, whose transversal spaces are its director spaces in $\Lambda$.

Suppose $n>h+1$. Let $\pi\notin\mathcal{P}$ be a $h$-subspace of $\Lambda$ with $\pi\cap L\ne\emptyset$, and $\Pi$ be an element of $\mathcal{P}$ such that $\Pi\cap\pi\ne\emptyset$.
Define $k=\dim(\Pi\cap\pi)\in\{0,\ldots,h-1\}$.
Then the subspace $\langle\Pi,\pi\rangle$ of $\Lambda$ has dimension $2h-k$. Together with Theorem \ref{th:bound} applied to the linear set $L\cap\langle\Pi,\pi\rangle$ in $\langle\Pi,\pi\rangle$, this implies
\[w_L(\langle\Pi,\pi\rangle)\leq \frac{(2h-k+1)n}{h+1}.\]
On the other hand $\langle L\cap\Pi,L\cap\pi\rangle$ is a subspace of $L\cap\langle\Pi,\pi\rangle$, and since $L$ is $m$-scattered for each $1\leq m\leq h$, we have
\[w_L(\langle\Pi,\pi\rangle)\geq w_L(\Pi)-w_L(\pi)-w_L(\Pi\cap\pi)\geq n+w_L(\pi)-(k+1).\]
Therefore,
\begin{equation}\label{eq:equaz}
    w_L(\pi)\leq\frac{(h-k)n}{h+1}+k+1.
\end{equation}
Since $n>h+1$, the right-hand side of \eqref{eq:equaz} attains its maximum for $k=0$; in this case, $w_L(\pi)\leq\frac{hn}{h+1}+1$, so that $w_L(\pi)<n$.
Thus, $\pi$ is not contained in any $h$-pseudoregulus for $L$, and the claim is proved.
\end{proof}

We are going to prove that the equivalence between two maximum $h$-scattered linear sets of $h$-pseudoregulus type only depends on the companion automorphisms associated with the $\fqn$-semilinear maps defining the line linear sets.
For any $\fq$-linear set $L_{\underline{f}}$ of $\Lambda$ as in \eqref{eq:linearpseudo}, we denote by $A(\underline{f})=(\mathrm{id},\sigma_2,\ldots,\sigma_{h+1})$ the $(h+1)$-ple of automorphisms associated with $\underline{f}$ and by $U_{A(\underline{f})}$ the $\fq$-vector space
\[ U_{A(\underline{f})}=\{(x,x^{\sigma_2},\ldots,x^{\sigma_{h+1}})\colon x\in \fqn\}. \]

\begin{theorem}\label{th:equivhscatt}
Two maximum $h$-scattered $\fq$-linear sets of $h$-pseudoregulus type $L_{\ff}$ and $L_{\underline{g}}$ in $\PG((h+1)t-1,q^n)$ are $\mathrm{P}\Gamma\mathrm{L}((h+1)t,q^n)$-equivalent if and only if $U_{A(\underline{f})}$ and $U_{A(\underline{g})}$ are $\Gamma\mathrm{L}(h+1,q^n)$-equivalent.
\end{theorem}
\begin{proof}
By Theorem \ref{th:simple}, $L_{\ff}$ and $L_{\underline{g}}$ are $\mathrm{P}\Gamma\mathrm{L}((h+1)t,q^n)$-equivalent if and only if $U_{\underline{f}}$ and $U_{\underline{g}}$ are $\Gamma\mathrm{L}((h+1)t,q^n)$-equivalent.
Also, by Theorem \ref{th:uniqPseudo}, the pseudoreguli $\mathcal{P}_\ff$ and $\mathcal{P}_{\underline{g}}$, associated respectively with $L_\ff$ and $L_{\underline{g}}$, are uniquely determined.

If the linear sets $L_{\ff}$ and $L_{\underline{g}}$ are ${\rm P\Gamma L}((h+1)t,q^n)$-equivalent, then by Theorem \ref{th:simple} there exists $F \in {\rm \Gamma L}((h+1)t,q^n)$ such that $F(U_\ff)=U_{\underline{g}}$.
Let $\Pi_{x_i}=\PG(W_{x_i},\fqn)$, with $i\in\{1,\ldots,t\}$, be $t$ elements of $\mathcal{P}_{\ff}$ such that
\[ U_\ff=(W_{x_1}\cap U_\ff)\oplus\ldots\oplus (W_{x_t}\cap U_\ff). \]
Since $F(U_\ff)=U_{\underline{g}}$ and $\dim_{\fq}(W_{x_i}\cap U_\ff)=n$, we have that $\varphi_F(\Pi_{x_i})=\PG(F(W_{x_i}),\fqn)\in \mathcal{P}_{\underline{g}}$ and hence
\[ U_{\underline{g}}=(F(W_{x_1})\cap U_{\underline{g}})\oplus\ldots\oplus (F(W_{x_t})\cap U_{\underline{g}}). \]
In particular, $F(W_{x_1}\cap U_\ff)=F(W_{x_1})\cap U_{\underline{g}}$.
By \eqref{eq:standard} we have that $W_{x_1}\cap U_\ff$ is ${\rm GL}(h+1,q^n)$-equivalent to $U_{A(\ff)}$ and $F(W_{x_1})\cap U_{\underline{g}}$ is $\Gamma\mathrm{L}(h+1,q^n)$-equivalent to $U_{A(\underline{g})}$. Then the claim follows.

Conversely, assume that there exists $G \in \Gamma\mathrm{L}(\fqnt^{h+1},\fqn)$ such that $G(U_{A(\underline{f})})=U_{A(\underline{g})}$.
As above, we can write 
\[ U_\ff=(W_{x_1}\cap U_\ff)\oplus\ldots\oplus (W_{x_t}\cap U_\ff), \]
and 
\[ U_{\underline{g}}=(W_{x_1}'\cap U_{\underline{g}})\oplus\ldots\oplus (W_{x_t}'\cap U_{\underline{g}}), \]
for some $\PG(W_{x_i},\fqn) \in \mathcal{P}_\ff$ and $\PG(W_{x_i}',\fqn) \in \mathcal{P}_{\underline{g}}$ with $x_1,\ldots,x_t \in \fqnt$ which are $\fqn$-linearly independent.
By \eqref{eq:standard}, for each $i \in\{1,\ldots,t\}$ there exist two invertible $\fqn$-linear maps $F_{x_i}\colon W_{x_i}\rightarrow W$ and $H_{x_i}\colon W_{x_i}'\rightarrow W$ such that
\[ F_{x_i}(W_{x_i}\cap U_\ff)=U_{A(\ff)}\,\,\text{and}\,\,H_{x_i}(W_{x_i}'\cap U_{\underline{g}})=U_{A(\underline{g})}, \]
where $W=\fqn^{h+1}$.
Define $\psi_{x_i}=H_{x_i}^{-1}\circ G \circ F_{x_i}\colon W_{x_i}\rightarrow W_{x_i}'$.
Since $V=W_{x_1}\oplus\ldots \oplus W_{x_t}=W_{x_1}'\oplus\ldots \oplus W_{x_t}'$, if $\psi_{x_i}=H_{x_i}^{-1}\circ G \circ F_{x_i}\colon W_{x_i}\rightarrow W_{x_i}'$, we may define the map
\[ \psi\colon V \rightarrow V,\quad \psi\left( \sum_{i=1}^t w_{i} \right)=\sum_{i=1}^t \psi_{x_i}(w_i),\qquad \textrm{with}\quad w_i\in W_{x_i}. \]
Clearly, $\psi\in{\rm \Gamma L}(V,\fqn)$ and
\[ \psi(U_\ff)=U_{\underline{g}}, \]
so that $L_\ff$ and $L_{\underline{g}}$ are ${\rm P\Gamma L}((h+1)t,q^n)$-equivalent.
\end{proof}

Let $\mathrm{id},\sigma_2,\ldots,\sigma_{h+1}$ be the companion automorphisms associated with $\ff$ and $I_{\ff}=\{i_1=0,i_2,\ldots,i_{h+1}\}$ be such that $x^{\sigma_{j}}=x^{q^{i_j}}$ for each $j \in \{2,\ldots,h+1\}$.

\begin{corollary}\label{cor:equivissue}
Two maximum $h$-scattered $\fq$-linear sets of $h$-pseudoregulus type $L_{\ff}$ and $L_{\underline{g}}$ in $\PG((h+1)t-1,q^n)$ are $\mathrm{P}\Gamma\mathrm{L}((h+1)t,q^n)$-equivalent if and only if
\[ I_{\ff}=I_{\underline{g}}+s=\{i+s \pmod{n} \colon i \in I_{\underline{g}}\}, \]
for some $s \in \{0,\ldots,n-1\}$.
\end{corollary}
\begin{proof}
By Theorem \ref{th:equivhscatt}, $L_\ff$ and $L_{\underline{g}}$ are $\mathrm{P}\Gamma\mathrm{L}((h+1)t,q^n)$-equivalent if and only if $U_{A(\underline{f})}$ and $U_{A(\underline{g})}$ are $\Gamma\mathrm{L}(h+1,q^n)$-equivalent.
By Theorem \ref{th:MESescatt}, $U_{A(\underline{f})}$ and $U_{A(\underline{g})}$ are $\Gamma\mathrm{L}(h+1,q^n)$-equivalent if and only if
the MRD-codes $\cC=\langle x,x^{\sigma_2},\ldots,x^{\sigma_{h+1}} \rangle_{\fqn}$ and $\cC'=\langle x,x^{\tau_2},\ldots,x^{\tau_{h+1}} \rangle_{\fqn}$ are equivalent, where
the $\sigma_i$'s and the $\tau_j$'s are the companion automorphisms of $\ff$ and $\underline{g}$, respectively.
From Theorem \ref{th:equivmon} the claim follows.
\end{proof}

Finally, we point out that the asymptotic result of Theorem \ref{th:DanyU} about Moore exponent sets, together with Theorem \ref{th:mainsec4} and Corollary \ref{cor:equivissue}, leads to an asymptotic structural classification of maximum $h$-scattered linear sets of pseudoregulus type.

\begin{theorem}\label{th:asymp}
Let $L$ be a maximum $h$-scattered $\fq$-linear of rank $nt$ in $\Lambda=\PG((h+1)t-1,q^n)$ of $h$-pseudoregulus type, with $L$ equivalent to some $L_\ff$ as in \eqref{eq:linearpseudo}.
Let $I\subseteq \{0,\ldots,n-1\}$ be the Moore exponent set associated with $L_\ff$, $j$ be the largest element of $I$, and
\[N=\begin{cases} 1 & \textrm{if}\quad h=1, \\ 4j+2 & \textrm{if}\quad h=2, \\ \frac{13}{3}j+2 & \textrm{if}\quad h>2.  \end{cases}\]
If $q>5$ and $n>N$, then there exists $s\in\{0,\ldots,n-1\}$ such that $\gcd(s,n)=1$ and $L$ is equivalent to
\[\{\langle(x,x^{q^s},x^{q^{2s}}\ldots,x^{q^{hs}})\rangle_{\fqn}\colon x\in\fqnt^*\}.\]
\end{theorem}

\begin{remark}
For any $s,d\in\{1,\ldots,n-1\}$ coprime with $n$, recall that two generalized Gabidulin codes $\mathcal{G}_{k,s}$ and $\mathcal{G}_{k,d}$ are equivalent if and only if $s\equiv\pm d\pmod n$; see \cite{LN2016,LTZ}.
From Corollary \ref{cor:equivissue} and Theorem \ref{th:asymp} follows that, if $n>N$ and $q>5$, then there are exactly $\varphi(n)/2$ orbits of $h$-scattered $\fq$-linear sets of $h$-pseudoregulus type in $\Lambda$ under ${\rm P\Gamma L}(\Lambda)$, where $\varphi$ is the Euler's totient function.
\end{remark}

\bigskip
\par\noindent Vito Napolitano, Olga Polverino, Giovanni Zini and Ferdinando Zullo\\
Dipartimento di Matematica e Fisica,\\
Universit\`a degli Studi della Campania ``Luigi Vanvitelli'',\\
Viale Lincoln 5,\\
I--\,81100 Caserta, Italy\\
{{\em \{vito.napolitano,olga.polverino,giovanni.zini,ferdinando.zullo\}@unicampania.it}}

\end{document}